\documentclass{article}
\usepackage{amssymb}
\title{{\bf Some properties of various graphs associated with finite groups }}
\author{ {\bf Xiaoyou Chen}$^{1,3}$, {\bf Alireza Moghaddamfar}$^{ 2, 3}$\\[0.3cm] and {\bf Mahsa Zohourattar}$^2$\\[0.3cm] 
{\em   $^1$College of Science, Henan University of Technology,}\\ {\em $450001$, Zhengzhou, China}\\[0.2cm]
{\em  $^2$Faculty of Mathematics, K. N. Toosi
University of Technology,}\\
{\em P. O. Box $16315$--$1618$, Tehran, Iran,}\
 \and   and \\[0.2cm]
 {\em  $^3$Department of Mathematical Sciences, Kent State
University,}\\ {\em   Kent, Ohio $44242$, United States of
America}\\[0.2cm]
{\em  E-mails}:  {\tt 
moghadam@kntu.ac.ir}, { and} {\tt  amoghadd@kent.edu}\\[0.3cm] }

\newenvironment{proof}{\noindent {\em {Proof}}.}{$\square$
\medskip}
\newtheorem{theorem}{Theorem}[section]
\newtheorem{definition}[theorem]{Definition}
\newtheorem{corollary}[theorem]{Corollary}

\newtheorem{proposition}[theorem]{Proposition}

\newtheorem{lm}[theorem]{Lemma}

\begin{document}
\newcommand{\f}{\frac}
\newcommand{\sta}{\stackrel}
\maketitle
\begin{abstract}
\noindent  In this paper we have investigated some properties of the power graph and  commuting graph associated with 
a finite group, using their tree-numbers. Among other things, it has been shown that the simple group 
$L_2(7)$ can be characterized through the tree-number of its power graph.  Moreover, the classification 
of groups with power-free decomposition is presented. Finally,  we  have obtained 
an explicit formula concerning the tree-number of commuting graphs
associated with the Suzuki simple groups. 
\end{abstract}


\def\thefootnote{ \ }
\footnote{{\em $2010$ Mathematics Subject Classification}: 05C25,  20D05, 20D06.}

\renewcommand{\baselinestretch}{1}

\section{Notation and Definitions}
We will consider finite undirected simple graphs $\Gamma=(V_\Gamma, E_\Gamma)$, where $V_\Gamma\ne \emptyset$ and $E_\Gamma$ are the vertex set and edge set of $\Gamma$, respectively.
A  {\em clique} (or a {\em complete set}) in $\Gamma$ is a subset of $V_\Gamma$ consisting of pairwise adjacent vertices (we do not require that it be a maximal complete set).
Especially, a complete graph is a graph in which the vertex set is a complete set. 
A {\em coclique} ({\em edgeless graph} or {\em independent set}) in $\Gamma$ is a set of pairwise nonadjacent vertices. The {\em independence number},  denoted by $\alpha(\Gamma)$, is the size of the largest coclique in $\Gamma$.
 
 A spanning tree for a graph $\Gamma$ is a subgraph of $\Gamma$ which is a tree and contains all the vertices of $\Gamma$.  The {\em tree-number} (or {\em complexity}) of a graph $\Gamma$, denoted by $\kappa (\Gamma)$, is the number of spanning trees of $\Gamma$ (0 if $\Gamma$ is disconnected), see \cite{sachs}. The famous Cayley formula shows that the complexity of the complete graph with $n$ vertices is given by $n^{n-2}$ (Cayley's formula).

In this paper, we shall be concerned with some graphs arising from {\em finite} groups. Two well known graphs associated with groups are commuting graphs and power graphs, as defined more precisely below.
Let $G$ be a finite group and $X$ a nonempty subset of $G$:
\begin{itemize}
\item The {\em power graph} $\mathcal{P}(G, X)$,  has
$X$ as its vertex set with two distinct elements of
$X$ joined by an edge when one is a {\em power} of the other.
\item The {\em commuting graph} $\mathcal{C}(G, X)$,  has
$X$ as its vertex set with two distinct elements of
$X$ joined by an edge when they {\em commute} in $G$.
\end{itemize}

Clearly,  power graph $\mathcal{P}(G, X)$ is a subgraph of $\mathcal{C}(G, X)$.
In the case $X=G$, we will simply write $\mathcal{C}(G)$ and $\mathcal{P}(G)$ instead of  $\mathcal{C}(G, G)$ and $\mathcal{P}(G, G)$, respectively.
Power and commuting graphs have been
considered in the literature, see for instance \cite{Kelarev, Britnell, Cameron, Das, Mahmoudifar,  Moghaddamfar}. In particular, in \cite[Lemma 4.1]{Mahmoudifar}, it is shown that $\mathcal{P}(G)=\mathcal{C}(G)$ if and only if $G$ is a cyclic group of prime power order, or a generalized quaternion $2$-group, or a Frobenius group with kernel a cyclic $p$-group and complement a
cyclic $q$-group, where $p$ and $q$ are distinct primes.
Obviously, when $1\in X$,  the power graph $\mathcal{P}(G, X)$ and the commuting graph $\mathcal{C}(G, X)$  are connected, and we can talk about the complexity 
of these graphs. For convenience, we put  $\kappa_{\cal P}(G, X)=\kappa (\mathcal{P}(G, X))$, $\kappa_{\cal P} (G)=\kappa (\mathcal{P}(G))$, $\kappa_{\cal C}(G, X)=\kappa (\mathcal{C}(G, X))$ and $\kappa_{\cal C} (G)=\kappa (\mathcal{C}(G))$. Also, instead of $\kappa_{\cal P}(G, X)$ and $\kappa_{\cal C}(G, X)$, we simply write  $\kappa_{\cal P}(X)$ and $\kappa_{\cal C}(X)$, if it does not lead to confusion.
All groups under discussion in this paper are finite and our group theoretic notation is mostly standard and follows that in \cite{atlas}. 

%%%%%%%%%%%%%%%%%%%%%%%%%%%%%%%%%%%%%%%%%%%%%%%%%%%%%%%%
\section{General Lemmas}
We first establish some notation which will be used repeatedly in the sequel.
Given a graph $\Gamma$, we denote by $\mathbf{A}_\Gamma$ and $\mathbf{D}_\Gamma$ the adjacency matrix and the diagonal matrix of vertex degrees of $\Gamma$, respectively. The Laplacian matrix of $G$ is defined
as $\mathbf{L}_\Gamma=\mathbf{D}_\Gamma-\mathbf{A}_\Gamma$. Clearly, $\mathbf{L}_\Gamma$ is a real symmetric matrix and its eigenvalues are nonnegative real numbers.
The Laplacian spectrum of $\Gamma$ is
$${\rm Spec}(\mathbf{L}_\Gamma)=\left(\mu_1(\Gamma), \mu_2(\Gamma), \ldots, \mu_n(\Gamma)\right),$$
where $\mu_1(\Gamma)\geqslant \mu_2(\Gamma)\geqslant \cdots\geqslant \mu_n(\Gamma)$,
are the eigenvalues of $L_\Gamma$ arranged in weakly decreasing order, and $n=|V(\Gamma)|$.
Note that,  $\mu_n(\Gamma)$ is  $0$, because each row sum of $\mathbf{L}_\Gamma$ is 0.
Instead of $\mathbf{A}_\Gamma$, $\mathbf{L}_\Gamma$, and $\mu_i(\Gamma)$  we simply write $\mathbf{A}$, $\mathbf{L}$,  and $\mu_i$  if it does not lead to confusion.

For a graph with $n$ vertices and Laplacian spectrum $\mu_1\geqslant \mu_2\geqslant \cdots\geqslant \mu_n=0$
it has been proved \cite[Corollary 6.5]{Biggs} that:
\begin{equation}\label{eq1}
\kappa(\Gamma)=\frac{\mu_1 \mu_2 \cdots \mu_{n-1}}{n}.
\end{equation}

The vertex--disjoint union of the graphs $\Gamma_1$ and $\Gamma_2$ is denoted by $\Gamma_1\oplus\Gamma_2$.
Define the {\em join} of $\Gamma_1$ and $\Gamma_2$ to be $\Gamma_1 \vee \Gamma_2=(\Gamma_1^c\oplus\Gamma_2^c)^c$. Evidently this is the graph formed from the vertex--disjoint union of the two graphs $\Gamma_1, \Gamma_2$,
by adding edges joining every vertex of  $\Gamma_1$  to every
vertex of $\Gamma_2$.
Now, one may easily prove the following (see also \cite{Merris}). 
\begin{lm}\label{elementary0}
Let $\Gamma_1$ and $\Gamma_2$ be two graphs on disjoint sets  with $m$ and $n$ vertices, respectively. If
$${\rm Spec}(\mathbf{L}_{\Gamma_1})=\left(\mu_1(\Gamma_1), \mu_2(\Gamma_1), \ldots, \mu_m(\Gamma_1)\right),$$
and
$${\rm Spec}(\mathbf{L}_{\Gamma_2})=\left(\mu_1(\Gamma_2), \mu_2(\Gamma_2), \ldots, \mu_n(\Gamma_2)\right),$$ then the following hold:
\begin{itemize}
\item[{\rm (1)}] The eigenvalues of Laplacian matrix  $\mathbf{L}_{\Gamma_1\oplus \Gamma_2}$ are:
$$\mu_1(\Gamma_1), \  \ldots, \ \mu_{m}(\Gamma_1),  \ \mu_1(\Gamma_2), \  \ldots, \ \mu_{n}(\Gamma_2).$$

\item[{\rm (2)}] The eigenvalues of Laplacian matrix  $\mathbf{L}_{\Gamma_1\vee \Gamma_2}$ are:
$$m+n, \ \mu_1(\Gamma_1)+n, \  \ldots, \  \mu_{m-1}(\Gamma_1)+n, \ \mu_1(\Gamma_2)+m, \ \ldots, \ \mu_{n-1}(\Gamma_2)+m, \ 0.$$
\end{itemize}
\end{lm}

\noindent {\em Two Examples.}  (1) Consider the complete bipartite graph $K_{a, b}=K_a^c\vee K_b^c$. 
Then, by Lemma \ref{elementary0} (2), the eigenvalues of Laplacian matrix  $\mathbf{L}_{K_a^c\vee K_b^c}$ are:
$$a+b, \ \underbrace{b, \ b,  \ldots, \  b}_{(a-1)-{\rm times}},  \underbrace{a, \ a,  \ldots, \ a}_{(b-1)-{\rm times}} , \ 0.$$
Using Eq. (\ref{eq1}) we get $\kappa(K_{a, b})=b^{a-1} a^{b-1}$.

(2)
A graph $\Gamma$ is a {\em split graph} if its vertex set  can be partitioned into a clique $C$ and an independent set $I$, where $V_\Gamma=C\uplus I$ is called a {\em split partition} of $\Gamma$ (see \cite{Foldes-Hammer}).  Now, consider the split graph $\Gamma=K_a\vee K_b^c$. Again, by Lemma \ref{elementary0} (2), the eigenvalues of Laplacian matrix  $\mathbf{L}_{K_a\vee K_b^c}$ are:
$$a+b, \ \underbrace{a+b, \  a+b, \ \ldots, \  a+b}_{(a-1)-{\rm times}},  \underbrace{a, \ a, \  \ldots, \ a}_{(b-1)-{\rm times}} , \ 0,$$
and it follows from Eq. (\ref{eq1}) that  $\kappa(\Gamma)=(a+b)^{a-1} a^{b-1}$.

Next lemma determines the complete commuting graphs and power graphs.
\begin{lm}\label{complete-CSS}
Let $G$ be a finite group. Then, we have 
\begin{itemize}
\item[{\rm (1)}]  The commuting graph $\mathcal{C}(G)$ is complete iff $G$ is an abelian group. 

\item[{\rm (2)}]  The power graph $\mathcal{P}(G)$ is complete iff $G$ is a cyclic $p$-group for some
prime $p$ {\rm (see Theorem 2.12 in   \cite{CSS})}.
\end{itemize}
\end{lm}

An immediate consequence of Lemma \ref{complete-CSS} is that 
$\kappa_{\cal P}(\mathbb{Z}_{p^n})=p^{n(p^n-2)}$ and if $X\subset G$ is a commuting set, then 
$\kappa_{\cal C}(G, X)=|X|^{|X|-2}$. We will use these facts without further references. 

\begin{lm}\label{elementary}  Let $M$ and $N$ be subgroups of a group $G$ such that $M\cap N=1$.
Let $x\in M$ and $y\in N$ be two arbitrary nontrivial elements. Then, $x$ and $y$ are nonadjacent in ${\cal P}(G)$ as two vertices. 
In particular, if $m>1$ and 
$X=\cup_{j=1}^{m}G_j$,
where $1<G_j<G$ and $G_i\cap G_j=1$ for $i\neq j$, then we have
$$\mathcal{P}(X^\#)=\bigoplus_{j=1}^{m}\mathcal{P}\big(G_j^\#\big),$$
where $X^\#=X\setminus \{1\}$  and $G_j^\#=G_j\setminus \{1\}$. 
\end{lm}
\begin{proof}  It is easy to see that 
$\langle x\rangle \cap \langle y\rangle\subseteq M\cap N=1$,
which forces $\langle x\rangle \not\subseteq \langle y\rangle$ and $\langle y\rangle \not\subseteq \langle x\rangle$.
\end{proof}

As an immediate consequence of Lemma \ref{elementary}, we obtain
\begin{corollary}\label{cor1}
Let $G$ be a finite group. Then the following hold:
\begin{itemize} 
\item[{\rm (1)}] If $G$ has even order, then ${\rm Inv}(G)$, the set  of involutions of $G$, forms an independent set of $\mathcal{P}(G)$.
\item[{\rm (2)}] Any pair of elements in $G$ with relatively prime orders forms an independent set of $\mathcal{P}(G)$, especially we have  $\alpha({\cal P}(G))\geqslant |\pi(G)|$, where $\pi(G)$ denotes 
the set of all prime divisors of $|G|$.
\end{itemize}
\end{corollary}

%%%%%%%%%%%%%%%%%%%%%%%%%%%%%%%%%%%%%%%%%%%%%%%%%

A {\em universal} vertex is a vertex of a graph that is adjacent to all other vertices of the graph.
The following lemma  \cite[Proposition $4$]{Cameron} determines the set of universal vertices of the power graph
of $G$, in general case.

\begin{lm}\label{universal} 
 Let $G$ be a finite group and $S$  the set of universal vertices of the power graph $\mathcal{P}(G)$. Suppose that $|S|>1$. Then one of the following occurs:
\begin{itemize}
\item[{\rm (a)}]  $G$ is cyclic of prime power order, and $S=G$;
\item[{\rm (b)}]  $G$ is cyclic of non-prime-power order $n$, and $S$ consists of the identity and the 
generators of $G$, so that $|S|=1+\phi(n)$;
\item[{\rm (c)}]  $G$ is generalized quaternion, and $S$ contains the identity and the unique involution in $G$, so that $|S|=2$.
\end{itemize}
\end{lm}

\begin{lm} \label{th-4}  {\rm \cite[Theorem 3.4]{MRSS}}
If $H_1, H_2, \ldots, H_t$ are nontrivial subgroups of a group $G$
such that $H_i\cap H_j=\{1\}$, for each
$1\leqslant i<j\leqslant t$, then we have 
$\kappa_{\cal P}(G)\geqslant\kappa_{\cal P}(H_1) \kappa_{\cal P}(H_2)\cdots\kappa_{\cal P}(H_t)$. 
\end{lm}

\begin{lm}\label{marty}  {\rm \cite[Lemma 6.1]{MRSS}}
Let $G$ be a finite nonabelian simple group and let $p$ be a
prime dividing the order of $G$. Then $G$ has at least $p^2-1$
elements of order $p$, or equivalently, there is at least $p+1$
cyclic subgroups of order $p$ in $G$.
\end{lm}

%%%%%%%%%%%%%%%%%%%%%%%%%%%%%%%%%%%%%%%%%%%%%%

\section{Main Results}

%%%%%%%%%%%%%%%%%%%%%%%%%%%%%%%%%%%%%%%%%%%%%%

\subsection{Power Graphs}
We begin with some elementary but useful results of power graphs. 
Before stating the results, we need to introduce some additional notation.
Let $X$ be
a nonempty subset of $G^\#$, the set of nonidentity elements of $G$.  We denote by $1_X$ 
the bipartite graph with partite sets $\{1\}$ and $X$.  
Let $\phi$ denote Euler's totient function, so that  $\phi(n)=|{\Bbb Z}_n^\times|$.
 We will preserve these notation throughout this section.
\begin{lm}\label{lm-1-frob}
Let $G$ be a group and $H$ be a proper subgroup of $G$. If $m$ is
the order of an element of $G\setminus H$, then we have
$$\kappa_{\cal P} (G)\geqslant \left(\phi(m)+1\right)^{\phi(m)-1}\kappa_{\cal P}(H).$$
In particular, $\kappa_{\cal P} (G)\geqslant \kappa_{\cal P}(H)$, with equality if
and only if $G$ is a Frobenius group with kernel $H$ and
complement $C$ of order $2$.
\end{lm}
\begin{proof}
Let $x$ be an element in $G\setminus H$ of order $m$. Set
$\Omega_x=\{y \ | \ \langle y\rangle=\langle x\rangle\}\cup \{1\}$.
Clearly, $\Omega_x\cap H=1$ and $|\Omega_x|=\phi(m)+1$. Note that the induced subgraph
$\mathcal{P}(G, \Omega_x)$  is a complete graph, that is
$\mathcal{P}(G, \Omega_x)=K_{\phi(m)+1}$. Now, if $T_{\Omega_x}$  and $T_H$
are two arbitrary spanning trees of $\mathcal{P}(G, \Omega_x)$ and
$\mathcal{P}(G, H)$, respectively, then $T_G=T_{\Omega_x} \cup T_H\cup
1_{G\setminus(\Omega_x\cup H)}$ is a spanning tree of $\mathcal{P}(G)$.
Thus, by product rule the number of such spanning trees of
$\mathcal{P}(G)$ is equal to
$$\kappa_{\cal P}(\Omega_x)\cdot \kappa_{\cal P}(H)\cdot 1=(\phi(m)+1)^{\phi(m)-1}\kappa_{\cal P}(H). \ \ \ \mbox{(by
Cayley's formula)}$$ This shows that the following inequality
holds: $$\kappa_{\cal P}(G)\geqslant (\phi(m)+1)^{\phi(m)-1}\kappa_{\cal P}(H),$$
as required. Finally, since for each positive integer $m$,
$(\phi(m)+1)^{\phi(m)-1}\geqslant 1$,  it follows that $\kappa_{\cal P}(G)\geqslant
\kappa_{\cal P}(H)$.

The preceding argument suggests how to construct a spanning tree
of $\mathcal{P}(G)$ through a spanning tree of $\mathcal{P}(G, H)$. In fact,
if $T_H$ is a spanning tree of $\mathcal{P}(G, H)$, then $T_G=T_H\cup
1_{G\setminus H}$ is a spanning tree of
$\mathcal{P}(G)$, which leads again to the inequality
$\kappa_{\cal P} (G)\geqslant \kappa_{\cal P}(H)$. Moreover, the equality holds if and only
if $\mathcal{P}(G, G\setminus H)$ is a null graph and there are no edges between vertices in $G\setminus H$ and $H^\#$. We argue under these conditions that $G$ is a Frobenius group
with kernel $H$ and complement $C$ of order $2$. We first observe
that every element in $G\setminus H$ is an involution. Also,
for all $x\in G\setminus H$ and $h\in H$, $xh\in G\setminus H$
and so  $(xh)^2=1$, or equivalently $x^{-1}hx=h^{-1}$. This
shows that $H$ is a normal subgroup of $G$ and the cyclic subgroup
$C=\langle x \rangle$ of order 2 acts on $H$ by conjugation which
induces a fixed-point-free automorphism of $H$. Hence, $G=HC$ is
a Frobenius group with kernel $H$ and complement $C$, as required.
\end{proof}

A group $G$ from a class $\mathcal{F}$ is said to be recognizable
in $\mathcal{F}$ by $\kappa_{\cal P} (G)$ (shortly, $\kappa_{\cal P} $-recognizable in
$\mathcal{F}$) if every group $H \in \mathcal{F}$ with $\kappa_{\cal P} 
(H)=\kappa_{\cal P}  (G)$ is isomorphic to $G$. In other words, $G$ is
$\kappa_{\cal P} $-recognizable in $\mathcal{F}$ if $h_{\mathcal{F}}(G)=1$,
where $h_{\mathcal{F}}(G)$ is the (possibly infinite) number of
pairwise non-isomorphic groups $H\in \mathcal{F}$ with
$\kappa_{\cal P} (H)=\kappa_{\cal P} (G)$. We denote by 
$\mathcal{S}$ the classes of all finite
simple groups. In the sequel, we show that the simple group $L_2(7)\cong L_3(2)$  is 
$\kappa_{\cal P} $-recognizable group in class $\mathcal{S}$, in other words  $h_{\mathcal{S}}(L_2(7))=1$. 

\begin{theorem}\label{th1-new} The simple group $L_2(7)$ is
$\kappa_{\cal P} $-recognizable in the class $\mathcal{S}$ of all finite simple groups,
that is, $h_{\mathcal{S}}(L_2(7))=1$.
\end{theorem}
\begin{proof}
Let $G\in \mathcal{S}$ with $\kappa_{\cal P} (G)=\kappa_{\cal P} (L_2(7))=2^{84}\cdot 3^{28}\cdot 7^{40}$ (see {\rm \cite[Theorem 4.1]{kmssz}}). 
We have to prove that $G$ is isomorphic to $L_2(7)$. 
Clearly, $G$ is nonabelian, 
since otherwise $G\cong \mathbb{Z}_p$ for some prime $p$, and so
 $\kappa_{\cal P} (G)=\kappa_{\cal P} (\mathbb{Z}_p)=p^{p-2}$, which is a
contradiction. Now, we claim that $\pi(G)\subseteq \{2, 3, 5, 7\}$.  Suppose $p\in \pi(G)$
and $p\geqslant 11$. Let $c_p$ be the number
of cyclic subgroups of order $p$ in $G$. By Lemma \ref{marty},
$c_p\geqslant p+1$, because $G$ is a nonabelian simple group.
Therefore, from Lemma \ref{th-4}, we deduce that
$$\kappa_{\cal P} (G)\geqslant \kappa_{\cal P} (\mathbb{Z}_p)^{c_p} \geqslant
\kappa_{\cal P} (\mathbb{Z}_p)^{p+1}=p^{(p-2)(p+1)}\geqslant
11^{108}>\kappa_{\cal P} (G),$$ which is a contradiction. This shows that 
$\pi(G)\subseteq \{2, 3, 5, 7\}$, as claimed.

By results collected in  \cite[Table 1]{zav}, $G$ is isomorphic
to one of the groups $A_5\cong L_2(4)\cong L_2(5)$, $A_6\cong L_2(9)$, $S_4(3)\cong U_4(2)$, 
$L_2(7)\cong L_3(2)$, $L_2(8)$, $U_3(3)$, $A_7$, $L_2(49)$, $U_3(5)$, 
$L_3(4)$, $A_8\cong L_4(2)$, $A_9$, $J_2$, $A_{10}$, $U_4(3)$,
$S_4(7)$, $S_6(2)$ or $O_8^+(2)$.
In all cases, except  $A_5$ and  $L_2(7)$, $G$ 
contains a subgroup $H$ which is isomorphic to $A_6$ (see
\cite{atlas}). But then, $\kappa_{\cal P}(G)\geqslant \kappa_{\cal P}(H)=2^{180}\cdot 3^{40}\cdot 5^{108}$,  a contradiction.
If $G$ is isomorphic to $A_5$, then $\kappa_{\cal P}(G)=2^{20}\cdot  3^{10}\cdot  5^{18}$,  which contradicts the assumption. Thus $G$  is isomorphic to $L_2(7)$,  as required.
\end{proof}

%%%%%%%%%%%%%%%%%%%%%%%%%%%%%%%%%%%%%%%%%%%%%%%%

\subsection{Power-Free Decompositions} 
A generalization of split graphs was introduced and investigated under the name $(m, n)$-graphs in \cite{Brandstadt}. 
 A graph $\Gamma$ is an $(m,n)$-graph if its vertex set can be partitioned into $m$ cliques $C_1, \ldots, C_m$ and $n$ independent sets $I_1, \ldots, I_n$.  In this situation, 
 $$V_\Gamma=C_1\uplus C_2\uplus\cdots\uplus C_m\uplus I_1\uplus I_2\uplus\cdots\uplus I_n,$$
 is called an $(m,n)$-split partition of $\Gamma$. Thus, $(m,n)$-graphs are a natural generalization of split graphs, which are precisely $(1,1)$-graphs. 

Accordingly, we are motivated to make the following definition.

\begin{definition}\label{def2} {\rm Let $G$ be a group and $n\geqslant 1$ an integer.
We say that $G$ has an {\em$n$-power-free decomposition} if it
can be partitioned as a disjoint union of  a cyclic $p$-subgroup $C$ of maximal order 
and $n$ nonempty subsets $B_1, B_2, \ldots, B_n$:
\begin{equation}\label{eq2} G=C\uplus B_1\uplus B_2\uplus \cdots \uplus B_n,\end{equation}
such that  the $B_i$'s are independent sets in ${\cal P}(G)$ and $|B_i|>1$, for each $i$. 
If $n=1$, we simply say $G=C\uplus B_1$ is a {\em power-free decomposition} of $G$.}
\end{definition}

Since $C$ is a cyclic $p$-subgroup  of maximal order in Definition \ref{def2}, $C$ is a clique, and so 
 Eq (\ref{eq2}) is a $(1,n)$-split partition of $\mathcal{P}(G)$.
Note that, there are some finite groups which do not have an $n$-power-free decomposition,
for any $n$, for example one can consider cyclic groups (see Proposition \ref{non-existence}).  On the other hand,  the structure of groups $G$ which have
a power-free decomposition is obtained (see Theorem \ref{th-1}).

\begin{lm}\label{lm-eee0}
Suppose $G$ has an $n$-power-free decomposition:
 $$G=C \uplus B_1 \uplus B_2\uplus \cdots \uplus B_n,$$ 
where $C$ is a cyclic $p$-subgroup of $G$. Then the following statements hold:
\begin{itemize}
\item[{\rm (a)}] If $b\in G\setminus C$, then $\phi(o(b))\leqslant n$. In particular, we have $$\pi(G)\subseteq \pi((n+1)!)\cup \{p\}.$$
\item[{\rm (b)}] If $p\notin \pi((n+1)!)$, then $C$ is  normal and $C_C(b)=1$
for each $b\in G\setminus C$. In particular, $Z(G)=1$.
\item[{\rm (c)}] The set  of universal vertices of $\mathcal{P}(G)$ is contained in $C$. 
\end{itemize}
\end{lm}
\begin{proof} (a)  The first  statement follows immediately from the fact that the set of  generators of cyclic group  $\langle b\rangle$, which has $\phi(o(b))$ elements, forms a complete set in the $\mathcal{P}(G, G\setminus C)$, and hence each $B_i$ contains at most one of the generators. The second statement is also clear, because for each $q\in \pi(G)\setminus \{p\}$,  there exists an element  $b\in G\setminus C$
of order $q$, and so by first part $\phi(q)=q-1\leqslant n$, or $q\leqslant n+1$. 

(b) Assume the contrary.  Let $C=\langle x\rangle$ with $o(x)=p^m>1$.  Then, there exists $b\in G\backslash C$ such that $x^b\notin C$. By part (a) it follows that $\phi(o(x^b))\leqslant n$. Since $o(x^b)=o(x)$, 
$\phi(o(x^b))=\phi(o(x))=\phi(p^m)=p^{m-1}(p-1)$, and so we obtain
$$p-1\leqslant p^{m-1}(p-1)\leqslant n.$$
This forces $p\leqslant n+1$, which contradicts the hypothesis.

Let $b\in G\setminus C$.  Suppose $c$ in $C$ is not the identity and commutes with $b$.
Replacing $b$ by an appropriate power, we may assume without loss that $p$ divides $o(bc)$.  
Thus, we conclude that $p-1$ divides $\phi(o(bc))$.  Since  
$bc\in G\setminus C$, by part (a) we have $\phi(o(bc))\leqslant n$.
Thus, it follows that $p-1\leqslant n$, which is a contradiction.  This shows that  $C_C(b)=1$, as required.

(c) It is clear from Definition \ref{def2}.
\end{proof}

As the following result shows that there are some examples of groups for which there does not exist any $n$-power-free decomposition.
\begin{proposition}\label{non-existence}
Any cyclic group has no  $n$-power-free decomposition.
\end{proposition}
\begin{proof}
Assume the contrary and let $G=\langle x\rangle$ be a cyclic group with an $n$-power-free decomposition:
$$G=C\uplus B_1 \uplus B_2\uplus \cdots \uplus B_n,$$  for some $n\geqslant 1$, 
where $C\subset G$ is a cyclic $p$-subgroup. Clearly, $x$ is a universal vertex in $\mathcal{P}(G)$, and so by Lemma \ref{lm-eee0} (c),  $x\in C$. But then $C=G$, which is a contradiction.  The proof is complete.
\end{proof}

\begin{proposition}\label{existence}
The generalized quaternion
group $Q_{2^n}$, $n\geqslant 3$, has a  $2$-power-free decomposition.
\end{proposition}
\begin{proof}  With the following presentation:
$$Q_{2^n}=\langle x, y \ | \ x^{2^{n-1}}=1, y^2=x^{2^{n-2}}, x^y=x^{-1}\rangle,$$
we may choose $C=\langle x\rangle$, and
$$ B_1=\{y, xy, \ldots, x^{2^{n-2}-1}y\},  \ B_2=\{x^{2^{n-2}}y,  x^{2^{n-2}+1}y, \ldots, x^{2^{n-1}-1}y\}.$$
Then  
$Q_{2^n}=C\uplus B_1 \uplus B_2$ is a $2$-power-free decomposition,  and this
completes the proof.
\end{proof}

Given a group $G$, $1\in G$ is a universal vertex of the power graph 
$\mathcal{P}(G)$. Now, as an immediate corollary of Lemma \ref{universal} and Propositions \ref{non-existence} and \ref{existence}, we have the following.

\begin{corollary}\label{universal-de} 
Let $G$ be a  group,  $S$  the set of universal vertices of the power graph $\mathcal{P}(G)$, and $|S|>1$.  Then $G$ has an $n$-power-free decomposition iff  
$G$ is isomorphic to a generalized quaternion
group. \end{corollary}

\begin{theorem}\label{th-1} The following conditions on a group $G$ are equivalent:
\begin{itemize}
\item[{\rm (a)}] $G$ has a power-free decomposition,  $G=C\uplus B$, where $C$ is a cyclic $p$-subgroup of $G$. 
\item[{\rm (b)}] One of the following statements holds:  
\begin{itemize}
\item[{\rm (1)}] $p=2$ and $G$ is an elementary abelian $2$-group of order $\geqslant 4$.
\item[{\rm (2)}] $p=2$ and $G$ is the dihedral group $D_{2^m}$ of order $2^m$, for some integer $m\geqslant 3$.
\item[{\rm (3)}] $p>2$ and $G$ is  the dihedral group $D_{2p^n}$ (a Frobenius group) of order $2p^n$ with a cyclic kernel of order $p^n$.
\end{itemize}
\end{itemize}
\end{theorem}
\begin{proof}  ${\rm (a)\Rightarrow(b)}$. Suppose $G=C\uplus B$ is a
power-free decomposition of $G$, where $C\subset G$ is a cyclic $p$-subgroup of maximal order. It follows by Lemma \ref{lm-eee0} (a) that every element $b\in B$ is an involution, and also $|G|=2^mp^n$, for some odd prime $p$ and $m\geqslant 1$, $n\geqslant 0$.
We shall treat the cases $n=0$  and $n\geqslant 1$, separately.

{\bf Case 1.}   {\em $n=0$.}  In this case, $G$ is a $2$-group. If $|C|\leqslant 2$, then $G$ is an elementary abelian $2$-group and (1) holds. We may now assume that $|C|>2$.  Put $C=\langle x\rangle$. 
Then, for every $b$ in $B$, $x^b$ is not an involution and so $x^b\in C$, which shows that $C$ is a normal subgroup of $G$.  Thus $G/C$
is an elementary abelian $2$-group by the previous paragraph. 

We now claim that $[G:C]=2$. To prove this, we  assume that $[G:C]=2^t$, where $t\geqslant 2$. 
Let $I={\rm Inv}(G)$ be the set of involutions of $G$. Then, we have  $I=B\cup \{z\}$, where $z$ is the unique involution in $C$, and so
$$|I|=|B|+1=|G|-|C|+1=|G|-\frac{|G|}{2^t}+1=\left(\frac{2^t-1}{2^t}\right)|G|+1\geqslant \frac{3}{4}|G|+1,$$
which  forces $G$ to be an abelian group. We recall that, a finite group is abelian if at least $3/4$ of its
 elements have order two. But then, if $b\in B$, then $bx$ is not an involution and also  $bx\notin C$, which is a contradiction. 
 
Let $b$ be an involution in $B$. Then $G=\langle x, b\rangle$.
Since $bx\in B$, $bx$ is an involution, and thus $bxb=x^{-1}$, which implies that  $G$ is a dihedral group
and (2) follows.

{\bf Case 2.}   {\em $n\geqslant 1$.}  In this case,  $|G|=2^mp^n$ where $m, n\geqslant 1$,  and $C$ is a cyclic $p$-group of maximal order. As in previous case $C=\langle x\rangle$ is a normal subgroup of $G$ and 
 $G/C$ is an elementary abelian $2$-group. Note that, $G$ does not contain an element of order $2p$, and so  $C=C_G(C)$. Moreover, since 
 $$G/C=N_G(C)/C_G(C) \hookrightarrow {\rm Aut}(C),$$
and ${\rm Aut}(C)$ is a cyclic group of order $\phi(p^n)=p^{n-1}(p-1)$, we conclude that $|G/C|=2$.  Therefore, if $b$ is an involution in $G$, then $G=\langle x, b\rangle=\langle x\rangle \rtimes \langle b\rangle$, and since $b$ acts on $\langle x\rangle$ fixed-point-freely,  $G$ is  a Frobenius group of order $2p^n$ with
 cyclic kernel $C$ of order $p^n$, and (3) follows.

${\rm (b)\Rightarrow(a)}$.  Obviously.
\end{proof}

%%%%%%%%%%%%%%%%%%%%%%%%%%%%%%%%%%%%%%%%%%%%%%%%%%%%

\subsection{Commuting Graphs}
In this section, we consider the problem of finding the tree-number of the commuting graphs associated with
a family of finite simple groups. The Suzuki groups ${\rm Sz}(q)$, an infinite series of simple groups of Lie type, were defined in \cite{Suzuki2, Suzuki} as subgroups of the groups  ${\rm L}_4(q)$, with $q = 2^{2n+1}\geqslant 8$. 
In what follows, we shall give an explicit formula for $\kappa_{\cal C}({\rm Sz}(q))$.
Let $G={\rm Sz}(q)$, where $q=2^{2n+1}$.  We begin with some well-known facts about the simple group $G$.
These results have been obtained by Suzuki \cite{Suzuki2, Suzuki}:
\begin{itemize}
\item[{\rm (1)}]  Let  $r= 2^{n+1}$. Then $|G|= q^2(q-1)(q^2+1)= q^2(q-1)(q-r+1)(q+r+1)$, and  $\mu(G)=\{4, q-1, q-r+1, q+r+1\}$. 
For convenience, we write  $\alpha_q=q-r+1$ and  $\beta_q=q+r+1$.

\item[{\rm (2)}] Let $P$ be a Sylow $2$-subgroup of $G$. 
Then  $P$ is a $2$-group of order
$q^2$ with ${\rm exp}(P)=4$, which is a TI-subgroup, and  $|N_G(P)|=q^2(q-1)$.

\item[{\rm (3)}] Let $A\subset G$ be a cyclic subgroup of order $q-1$. Then $A$ is a TI-subgroup and
the normalizer $N_G(A)$ is a dihedral group of order $2(q-1)$.

\item[{\rm (4)}] Let $B\subset G$ be a cyclic subgroup of order $\alpha_q$. Then $B$ is a TI-subgroup and
the normalizer $N_G(B)$  has order $4\alpha_q$.

\item[{\rm (5)}] Let $C\subset G$ be a cyclic subgroup of order $\beta_q$. Then $C$ is a TI-subgroup and
the normalizer $N_G(C)$ has order $4\beta_q$.
\end{itemize}

We recall that, in general,  a subgroup  $H\leqslant G$  is a {\em TI-subgroup} (trivial intersection subgroup) if for every $g\in G$, either $H^g=H$ or $H\cap H^g=\{1\}$.

\begin{lm} \label{lem4.1}
 $\kappa_{\cal C}(P)=2^{(q-1)^2}q^{(q^2+q-3)}$.
\end{lm}

\begin{proof}
By Theorem VIII.7.9 of \cite{hupII} and Lemma XI.11.2 of \cite{hupIII}, $Z(P)$ is an elementary abelian $2$-group of order $q$ and every element outside $Z(P)$ has order $4$.  Observe that $P$ is the centralizer in $G$ of all of the nontrivial elements of $Z(P)$.  If $x \in P \setminus Z(P)$, then $\langle Z(P), x \rangle \leqslant C_G (x)$.  In the proof of Lemma XI.11.7 of \cite{hupIII}, we see that the elements of order $4$ in $G$ lie in two conjugacy classes. This implies that $|C_G (x)| = 2 |Z(P)|$, from which we deduce that $C_G (x) = \langle Z(P), x \rangle$. 
Then
for all $x, y\in P\setminus Z(P)$ either $C_G(x) = C_G(y)$ or $C_G(x)\cap C_G(y) = Z(P)$.
Hence, $\{C_G(x) | x\in P\setminus Z(P)\}$ forms a partition of $P$ for which
 the intersection of pairwise centralizers is $Z(P)$. This shows that 
$${\cal C}(P)=K_q\vee (\underbrace{K_{q}\oplus K_{q}\oplus \cdots \oplus K_{q}}_{q-1}). $$
Moreover, by Lemma \ref{elementary0}, the eigenvalues of Laplacian matrix  $L_{{\cal C}(P)}$ are:
$$q^2, \underbrace{q^2, \ q^2, \ \ldots, \ q^2}_{q-1}, \  \underbrace{2q, \ 2q,  \ \ldots, \ 2q}_{(q-1)^2}, \  \underbrace{q, \ q, \ \ldots, \ q}_{q-2}, \ 0. $$
It follows immediately using Eq. (\ref{eq1}) that
$$\kappa_{\cal C}(P)=2^{(q-1)^2}q^{(q^2+q-3)},$$
as required.  \end{proof}

\begin{theorem} \label{th4.2} Let  $q=2^{2n+1}$, where $n\geqslant 1$ is an integer. Then, we have  
$$\kappa_{\cal C}({\rm Sz}(q))=\left(2^{(q-1)^2}q^{(q^2+q-3)}\right)^{q^2+1}(q-1)^{(q-3)a}
({\alpha_q})^{(\alpha_q-2)b} (\beta_q)^{(\beta_q-2)c},$$
where $a=q^2(q^2+1)/2$, $b=q^2(q-1) \beta_q/4$ and $c=q^2(q-1) \alpha_q/4$.

\end{theorem}
\begin{proof} Let $G={\rm Sz}(q)$, where $q=2^{2n+1}\geqslant 8$. As already mentioned,  $G$ contains a Sylow 2-subgroup $P$ of order $q^2$ and cyclic subgroups $A$, $B$, and $C$, of orders
$q-1$, $\alpha_q$ and  $\beta_q$, respectively.
Moreover, every two  distinct conjugates of them intersect trivially and every element
of $G$ is a conjugate of an element in $P\cup A\cup B\cup C$. 
Looking at the proof of Lemma 11.6, we see that the cyclic subgroups $A$, $B$, and $C$,  are the centralizers of their nonidentity elements, while $P$ is the centralizer in $G$ of all of the nontrivial elements of $Z(P)$. 
Let 
$$\begin{array}{lllll} G & = & N_Px_1\cup \cdots \cup N_Px_p
& = &  N_{A}y_1\cup \cdots \cup N_{A}y_a\\[0.3cm] 
& = & N_{B}z_1\cup \cdots \cup N_{B}z_b
& = & N_{C}t_1\cup \cdots \cup N_{C}t_c,\\[0.3cm] 
\end{array}$$
be coset decompositions of $G$ by  $N_P=N_G(P)$, $N_A=N_G(A)$, $N_B=N_G(B)$ and $N_C=N_G(C)$, where
$p=[G:N_P]=q^2+1$, $a=[G:N_A]=q^2(q^2+1)/2$, $b=[G:N_B]=q^2(q-1) \beta_q/4$ and 
$c=[G:N_C]=q^2(q-1) \alpha_q/4$.
Then, we have
$$G=P^{x_1}  \cup \cdots \cup P^{x_p}\cup A^{y_1}\cup \cdots \cup A^{y_a}\cup  B^{z_1} \cup \cdots \cup B^{z_b}\cup C^{t_1} \cup \cdots \cup C^{t_c}.$$
 This shows that 
$$\begin{array}{lll}
{\cal C}(G) & = & K_1 \vee \left(p \ {\cal C}(P^\#)\oplus a \ {\cal C}(A^\#)\oplus b  \ {\cal C}(B^\#)\oplus c  \ {\cal C}(C^\#) \right)\\[0.3cm] 
& = & K_1 \vee \left(p \  {\cal C}(P^\#)\oplus a K_{(q-1)-1}\oplus b K_{\alpha_q-1}\oplus c K_{\beta_q-1}\right), \\[0.3cm] 
\end{array}$$
and so 
$$\kappa_{\cal C}(G)=\kappa_{\cal C}(P)^p\cdot  \kappa_{\cal C}(K_{q-1})^a\cdot \kappa_{\cal C}(K_{\alpha_q})^b\cdot \kappa_{\cal C}(K_{\beta_q})^c.$$
Now, Lemma \ref{lem4.1} and  Cayley's formula  yield the result. \end{proof}

%%%%%%%%%%%%%%%%%%%%%%%%%%%%%%%%%%%%%%%%%%%%%%

\begin{center}
 {\bf Acknowledgments }
\end{center}
This work was done during the first and second authors had a visiting position at the
Department of Mathematical Sciences, Kent State
University, USA. They would like to thank the hospitality of the Department of Mathematical Sciences of KSU. 
The first author thanks the funds (2014JCYJ14, 17A110004, 20150249, 20140970, 11571129).

\end{document}